\documentclass[reqno]{amsart}

\usepackage{amssymb}
\usepackage{graphicx}
\usepackage[usenames, dvipsnames]{color}
\usepackage{verbatim}
\usepackage{mathrsfs}
\usepackage{esint}
\usepackage[colorlinks,linkcolor=red,citecolor=blue]{hyperref}

\numberwithin{equation}{section}

\newtheorem{theorem}{Theorem}[section]
\newtheorem{corollary}[theorem]{Corollary}
\newtheorem{lemma}[theorem]{Lemma}
\newtheorem{prop}[theorem]{Proposition}
\newtheorem{remark}[theorem]{Remark}

\theoremstyle{definition}

\makeatletter
\def\dashint{\operatorname%
{\,\,\text{\bf-}\kern-.98em\DOTSI\intop\ilimits@\!\!}}
\makeatother

\def\\det{\text{det}}

\def\.5{\frac{1}{2}}

\newcommand{\RN}[1]{%
  \textup{\uppercase\expandafter{\romannumeral#1}}%
}

\newcommand{\dist}{\text{dist}}

\renewcommand{\epsilon}{\varepsilon}

\newcounter{marnote}

\begin{document}
\title[An Optimization Problem in Heat Conduction]
{An Optimization Problem in Heat Conduction With Volume Constraint and Double Obstacles}

\author[X. Li]{Xiaoliang Li}
\address[X. Li]{School of Mathematical Sciences, Laboratory of Mathematics and Complex Systems, MOE, Beijing Normal University, Beijing, 100875, China.}
\email{xiaoliangli@mail.bnu.edu.cn}

\author[C. Wang]{Cong Wang}
\address[C. Wang]{School of Mathematical Sciences, Laboratory of Mathematics and Complex Systems, MOE, Beijing Normal University, Beijing, 100875, China.}
\email{cwang@mail.bnu.edu.cn}



\begin{abstract}
We consider the optimization problem of minimizing $\int_{\mathbb{R}^n}|\nabla u|^2\,\mathrm{d}x$ with double obstacles $\phi\leq u\leq\psi$ a.e. in $D$ and a constraint on the volume of $\{u>0\}\setminus\overline{D}$, where $D\subset\mathbb{R}^n$ is a bounded domain. By studying a penalization problem that achieves the constrained volume for small values of penalization parameter, we prove that every minimizer is $C^{1,1}$ locally in $D$ and Lipschitz continuous in $\mathbb{R}^n$ and that the free boundary $\partial\{u>0\}\setminus\overline{D}$ is smooth. Moreover, when the boundary of $D$ has a plane portion, we show that the minimizer is $C^{1,\frac{1}{2}}$ up to the plane portion.
\end{abstract}

\maketitle


\section{Introduction and main results}

In this paper we study a variational problem with double obstacles and a volume constraint, which naturally arises in the study of optimal thermal insulation. More precisely, the physical motivation of our study comes from an optimal design problem in heat conduction which contains a model that may briefly be described as follows:

There is a room with two positive temperature profiles while its exterior environment temperature is fixed to be zero. One is led to keep the temperature inside the room between two given profiles in a way that minimizes the energy, by surrounding a prescribed volume of insulating material outside the room.

Such an optimization problem with volume constraint arising in heat conduction has attracted the interest of many authors, see for instance \cite{AAC 1986,ACS 1988,FMW 2006,FRW 2005,L 1996,T 2005,TU2017,Y 2016}. As far as we know, the volume constrained problems of double obstacle type have not been considered so far.

Mathematically, let $D$ be a smooth and bounded domain in $\mathbb{R}^n$ ($n\geq 2$), and let $\phi,\psi\in C^{2}(\overline{D})$ be two positive functions satisfying
$$\phi\leq\psi~\text{in}~D\quad\text{and}\quad\phi<\psi~\text{on}~\partial D.$$
Given $\mu>0$, we are concerned with the following problem:
\begin{equation}\label{eq:P}
\text{minimize}~J(u)=\int_{\mathbb{R}^n}|\nabla u|^2\,\mathrm{d}x\quad\text{over}~K_\mu,\tag{$P$}
\end{equation}
where
$$K_{\mu}=\left\{u\in H^1(\mathbb{R}^n)\mid\phi\leq u\leq\psi~\text{a.e.~in}~D,~\mathcal{L}^n(\{u>0\}\setminus\overline{D})=\mu\right\},$$
and $\mathcal{L}^n$ denotes the $n$-dimensional Lebesgue measure. In this setting, $\phi$ and $\psi$ are called lower and upper obstacles, respectively; the set $\partial\{u>0\}\setminus\overline{D}$ is a free boundary which determines the distribution of insulating material in the  above physical model. We also mention that problem \eqref{eq:P} is related to the optimal interface problem in fluids; see for instance \cite{AFMT 1999,LT 2006}.

The aim of this paper is to prove the existence of solutions to problem \eqref{eq:P}, and to study the regularity of the solutions as well as that of the corresponding free boundary. In this regard, when the upper obstacle $\psi$ is removed and the lower obstacle $\phi$ is required to be compactly supported in $D$, problem \eqref{eq:P} was studied in \cite{Y 2016}. Here we are considering the double obstacle case in a more general setting where the obstacles are allowed not to vanish on $\partial D$. Accordingly, differently from the single obstacle case treated in \cite{Y 2016}, we will encounter twofold difficulties respectively from double obstacle type problems and optimization problems with volume constraint. For this reason, we introduce below related researches on these two topics separately.

Over the past several decades, the single obstacle problems have been extensively studied, see for instance \cite{BK 1973/74,C 1980,C 1988,FS 2019,SY 2021,SY 2021A} and the references therein. However, less results are known for double obstacle problems. They are more complicated to deal with than the single case since the obstacles may touch each other and the geometry of their contact part could be rather bad. We send readers some existing literatures about that. Devoted to regularity of solutions, Dal Maso, Mosco and Vivaldi \cite{DMV 1989} studied the pointwise regularity that the solutions were proved to be continuous if the obstacles satisfy suitable Weiner-type conditions. Afterwards, the H\"{o}lder regularity results of $C^{0,\alpha}$ and $C^{1,\alpha}$ with $0<\alpha<1$ were obtained under various regularity assumptions on obstacles, see for instance \cite{C 1992,L 1991 India,MZ 1991}. In particular, for classical double obstacle problems, Chipot \cite{C 1979} and then Caffarelli and Kinderlehrer \cite{CK 1980} established the optimal regularity of $C^{1,1}$. To the best of our knowledge, up to now very little is known about the regularity of solutions when the obstacles are Lipschitz continuous. Furthermore, general regularity results of the free boundary have not been addressed yet.

Concerning optimization problems under volume constraint, the studies go back to the seminal paper \cite{AAC 1986}. There, Aguilera, Alt and Caffarelli presented a penalization method to establish the existence of minimizers and their regularity properties. This method has been applied to many optimal design problems of such type, see for instance \cite{ACS 1988,FMW 2006,FRW 2005,L 1996,T 2005}. However, in all these literatures, the Dirichlet boundary conditions are imposed on the problems considered there, which are not required in our problem \eqref{eq:P}.

In the present paper, we first employ penalization method to solve problem \eqref{eq:P} in a way of performing nonvolume preserving variations, which is in spirit of \cite{AAC 1986}. Precisely, for $0<\varepsilon<1$, we introduce the following penalization problem instead:
\begin{equation}\label{eq:P-e}
\text{minimize}~J_\varepsilon(u)=\int_{\mathbb{R}^n}|\nabla u|^2\,\mathrm{d}x+f_{\varepsilon}(\mathcal{L}^n(\{x\in\Omega\mid u(x)>0\}))\quad\text{over}~K,\tag{$P_\varepsilon$}
\end{equation}
where $\Omega=\mathbb{R}^n\setminus\overline{D}$,
$$K=\left\{u\in H^1(\mathbb{R}^n)\mid\phi\leq u\leq\psi~\text{a.e.~in}~D\right\},$$
and $f_{\varepsilon}$ is defined as
\begin{equation*}
f_{\varepsilon}(t)=
\begin{cases}
\varepsilon(t-\mu),&\text{if}~t\leq\mu,\\
\frac{1}{\varepsilon}(t-\mu),&\text{if}~t\geq\mu.\\
\end{cases}
\end{equation*}
We will show that problem \eqref{eq:P-e} can recover the original problem \eqref{eq:P} for small $\varepsilon$. Hence, it suffices to establish the existence of solutions to \eqref{eq:P-e}. Since \eqref{eq:P-e} is defined in the whole space $\mathbb{R}^n$, the existence of its solutions is not immediate. For this reason, we adapt the ideas in \cite{L 1996} to introduce further a family of solvable problems defined in balls $B_R$ (see \eqref{eq:P-e-R} in Section \ref{sec:2}) which are equivalent to \eqref{eq:P-e} for large $R$. This ultimately leads to the existence of solutions to original problem \eqref{eq:P}. Meanwhile, as a byproduct, we find the fact that any solution to problem \eqref{eq:P} has a bounded support. Physically, this agrees with the intuition when we go back to the model described at the beginning.

With the help of problems \eqref{eq:P-e} and \eqref{eq:P-e-R}, we also establish similar regularity results to those in \cite{Y 2016}. Indeed, we obtain $C^{1,1}$-regularity of the solutions inside $D$ as well as optimal Lipschitz regularity in the whole space $\mathbb{R}^n$. And the free boundary is proved to be a locally smooth surface, except for a closed set of $(n-1)$-dimensional Hausdorff measure zero. More precisely, our first main result is the following.
\begin{theorem}\label{thm:main 1}
There exists a solution to problem \eqref{eq:P}. Any solution $u$ is $C^{1,1}$ locally in $D$ and Lipschitz continuous in $\mathbb{R}^n$. And the free boundary $\partial\{u>0\}\setminus\overline{D}$ is smooth except for a closed set of $\mathcal{H}^{n-1}$-measure zero.
\end{theorem}

Since the obstacles $\phi$ and $\psi$ work on $\partial D$ in our case unlike the one considered in \cite{Y  2016}, we are able to investigate further the behavior of the solution $u$ near $\partial D$. Indeed, by considering the flat case, we derive the following improved regularity up to $\partial D$ from both inside and outside of $D$.

\begin{theorem}\label{thm:main 2}
Suppose that the boundary of $D$ has a plane portion $\Gamma$. Let $u$ be a solution to problem \eqref{eq:P}. Then
$$u\in C^{1,\frac{1}{2}}(D\cup\Gamma),$$
and for a small neighborhood $\mathcal{N}$ of $D$,
$$u\in C^{1,\frac{1}{2}}((\mathcal{N}\setminus\overline{D})\cup\Gamma).$$
\end{theorem}

It would be worthwhile to point out that the argument we exploit here to prove above regularity results skips the multiple parameters used in \cite{Y 2016}, and it enables us to derive Lipschitz regularity for classical double obstacle problem where the obstacles are Lipchitz continuous, see Remark \ref{rk:class-d-obstacle} for details. As already mentioned, such a regularity result was missing in the literature. Besides, we remark that since there may have a jump in $\nabla u$ across $\partial\{u>0\}\setminus\overline{D}$, the Lipschitz continuity is expected to be globally optimal as stated in \cite{AAC 1986,L 1996} for other optimization problems under volume constraint.

Now let us comment the proofs of Theorems \ref{thm:main 1} and \ref{thm:main 2}. The main idea is to split the problem into interior and exterior problems and especially to investigate the behavior of solutions near the boundary $\partial D$.
By illustrating that the set $\partial\{u>0\}$ stays away from $\partial D$, we carefully enlarge the obstacles $\phi$ and $\psi$ to be piecewise smooth in a small neighborhood $D_\delta$ of $D$ and to be Lipschitz continuous across $\partial D$. It is thus observed that $u$ still solves this extended double obstacle problem in $D_\delta$. Then by reducing that problem locally to a single obstacle problem and employing approximation arguments, we obtain the regularity of $u$ near $\partial D$. Finally, together with the fine properties gained from the penalization problem \eqref{eq:P-e}, we deduce the desired regularity of $u$ and that of the free boundary.

The paper is organized as follows. In Section \ref{sec:2}, we introduce a family of auxiliary problems \eqref{eq:P-e-R} where the admissible functions are supported in $B_R$. We show the existence and some fine properties of a solution to problem \eqref{eq:P-e-R}. Then in Section \ref{sec:3}, we further perform the analysis as in Section \ref{sec:2} to study problem \eqref{eq:P-e}, establishing qualitative results analogous to those for problem \eqref{eq:P-e-R}. We also derive the regularity properties of the free boundary. In Section \ref{sec:4}, we prove that problem \eqref{eq:P} is recovered from \eqref{eq:P-e} for small values of $\varepsilon$. At last in Section \ref{sec:5}, by illustrating the solution also solves a double obstacle problem near $\partial D$ where obstacles are piecewise smooth, we obtain the regularity of $C^{0,1}$ in $\mathbb{R}^n$ and $C^{1,\frac{1}{2}}$ up to $\partial D$. This, together with the local properties established in the previous sections, completes the proofs of Theorem \ref{thm:main 1} and Theorem \ref{thm:main 2}.

We fix some notations throughout this paper. For any $x=(x_1,x_2,\cdot\cdot\cdot,x_n)\in\mathbb{R}^n$ and $r>0$, define
$$B_r(x)=\{y\in\mathbb{R}^n\mid |y-x|<r\},\quad B'_r(x)=\{y\in B_r(x)\mid y_n=x_n\},$$
$$B^+_r(x)=\{y\in B_r(x)\mid y_n>x_n\},\quad B^-_r(x)=\{y\in B_r(x)\mid y_n<x_n\}.$$
If $x=0$, we would omit it and write $B_r$, $B'_r$ and $B_r^{\pm}$ for simplicity.

\section{Solutions to problem \eqref{eq:P-e-R}}\label{sec:2}

In this section, we introduce a family of auxiliary problems \eqref{eq:P-e-R} for fixed $0<\varepsilon<1$ by following the idea of \cite{L 1996}. We will show the existence of a solution $u_{\varepsilon,R}$ to problem \eqref{eq:P-e-R} and also present certain fine properties of $u_{\varepsilon,R}$, including Lipschitz regularity, nondegeneracy and density estimate. Moreover, we prove that the set $\{u_{\varepsilon,R}>0\}$ is bounded independently on $R$ through some estimates on the positive phase of $u_{\varepsilon,R}$ and Vitali covering lemma. These results will be used to study problem \eqref{eq:P-e} in the next section.

Take a positive number $R_0$ such that
\begin{equation}\label{R_0}
\mathcal{L}^n(B_{R_0}\setminus\overline{D})>\mu\quad\text{and}\quad\{x\in\mathbb{R}^n\mid\dist(x,D)<1\}\subset B_{R_0}.
\end{equation}
For $R\geq R_0$, let us consider an auxiliary version of the penalization problem \eqref{eq:P-e} as follows:
\begin{equation}\label{eq:P-e-R}
\text{minimize}~J_{\varepsilon,R}(u)=\int_{B_R}|\nabla u|^2\,\mathrm{d}x+f_{\varepsilon}(\mathcal{L}^n(\{x\in\Omega_R\mid u(x)>0\}))\quad\text{over}~K_R, \tag{$P_{\varepsilon,R}$}
\end{equation}
where $\Omega_R=B_R\setminus\overline{D}$ and
$$K_R=\left\{u\in H_0^1(B_R)\mid\phi\leq u\leq\psi~\text{a.e.~in}~D\right\}.$$

\begin{lemma}\label{lem:P-e-R,existence}
There exists a solution to problem \eqref{eq:P-e-R}. Moreover, if $u_{\varepsilon,R}$ is such a solution, then
$$0\leq u_{\varepsilon,R}\leq\sup\limits_D\phi\quad\mathrm{a.e.~in}~B_R.$$
\end{lemma}
\begin{proof}
We first prove that $J_{\varepsilon,R}$ is not always infinite in $K_R$. Let $v_0\in H^1(D)$ be the minimizer of the energy functional $\int_{D}|\nabla u|^2\,\mathrm{d}x$ over
$$K_D:=\left\{u\in H^1(D)\mid\phi\leq u\leq\psi~\text{a.e.~in}~D,~u=\frac{\phi+\psi}{2}~\text{on}~\partial D\right\}.$$
Its existence is guaranteed by direct variational method. Take a smooth and bounded domain $D'$ with
$$D\Subset D'\Subset B_R\quad\text{and}\quad\mathcal{L}^n(D'\setminus\overline{D})=\mu.$$
We extend $v_0$ by taking an auxiliary function $w_0$ such that
\begin{equation*}
\begin{cases}
\Delta w_0=0,\quad&\text{in}~D'\setminus\overline{D},\\
w_0=v_0,\quad&\text{on}~\overline{D},\\
w_0=0,\quad&\text{on}~\overline{B}_R\setminus D'.
\end{cases}
\end{equation*}
Then $w_0\in K_R$. Moreover, the strong maximum principle implies
$$w_0>0\quad\text{a.e.~in}~D'.$$
By the definition of $f$, we have
$$f_{\varepsilon}(\mathcal{L}^n(\{x\in\Omega_R\mid w_0(x)>0\}))=f_{\varepsilon}(\mathcal{L}^n(D'\setminus\overline{D}))=f_{\varepsilon}(\mu)=0.$$
It follows that
$$J_{\varepsilon,R}(w_0)=\int_{D'}|\nabla w_0|^2\,\mathrm{d}x<\infty$$
holds uniformly with respect to $0<\varepsilon<1$ and $R\geq R_0$.

Clearly, $J_{\varepsilon,R}\geq-\mu$ in $K_R$. We set
$$m=\inf\{J_{\varepsilon,R}(u)\mid u\in K_R\}<\infty.$$
Next, we prove $m$ is attained. There exists a minimizing sequence $\{u_k\}\subset K_R$ such that
$$J_{\varepsilon,R}(u_k)\rightarrow m,\quad\text{as}~k\rightarrow\infty.$$
This implies $\{J_{\varepsilon,R}(u_k)\}$ is a bounded sequence, then $\|\nabla u_k\|_{L^2(B_R)}$ is uniformly bounded. Hence, up to a subsequence,
$$u_{k}\rightarrow u_{\varepsilon,R}\quad\text{weakly~in}~H_0^1(B_R),\quad u_{k}\rightarrow u_{\varepsilon,R}\quad\text{a.e.~in}~B_R,\quad\text{as}~k\rightarrow\infty.$$
From this,
$$\phi\leq u_{\varepsilon,R}\leq \psi\quad\text{a.e.~in}~D,$$
\begin{equation*}
\mathcal{L}^n(\{x\in\Omega_R\mid u_{\varepsilon,R}(x)>0\})\leq\liminf_{k\rightarrow\infty}\mathcal{L}^n(\{x\in\Omega_R\mid u_{k}(x)>0\}),
\end{equation*}
and
$$\int_{B_R}|\nabla u_{\varepsilon,R}|^2\,\mathrm{d}x\leq\liminf_{k\rightarrow\infty}\int_{B_R}|\nabla u_k|^2\,\mathrm{d}x.$$
We thus get $u_{\varepsilon,R}\in K_R$ and
$$J_{\varepsilon,R}(u_{\varepsilon,R})\leq\liminf_{k\rightarrow\infty}J_{\varepsilon,R}(u_k)=m.$$
Therefore, $u_{\varepsilon,R}$ is a minimizer of $J_{\varepsilon,R}$ over $K_R$.

We proceed to show the $\mathcal{L}^{\infty}$-estimate on $u_{\varepsilon,R}$. Define for $0<t<1$,
$$u_t=u_{\varepsilon,R}-t\min\{u_{\varepsilon,R},0\},$$
$$u^t=u_{\varepsilon,R}+t\min\{\sup_D\phi-u_{\varepsilon,R},0\}.$$
It is easy to check that $u_t\in K_R$. Also,
$$u_t=u_{\varepsilon,R}>0,\quad\text{if}~u_{\varepsilon,R}>0,$$
$$u_t=(1-t)u_{\varepsilon,R}\leq0,\quad\text{if}~u_{\varepsilon,R}\leq0.$$
Hence we have
$$\mathcal{L}^n(\{x\in\Omega_R\mid u_t(x)>0\})=\mathcal{L}^n(\{x\in\Omega_R\mid u_{\varepsilon,R}(x)>0\}).$$
By the minimality of $u_{\varepsilon,R}$, we have
$$\int_{B_R}|\nabla u_{\varepsilon,R}|^2\,\mathrm{d}x\leq\int_{B_R}|\nabla u_{t}|^2\,\mathrm{d}x.$$
Using the first variation, we get
\begin{equation*}
\begin{split}
0&\leq\int_{B_R}|\nabla u_t|^2\,\mathrm{d}x-\int_{B_R}|\nabla u_{\varepsilon,R}|^2\,\mathrm{d}x\\
&=\int_{B_R\cap\{u_{\varepsilon,R}<0\}}\left(|\nabla u_t|^2-|\nabla u_{\varepsilon,R}|^2\right)\,\mathrm{d}x\\
&=t(t-2)\int_{B_R\cap\{u_{\varepsilon,R}<0\}}|\nabla u_{\varepsilon,R}|^2\,\mathrm{d}x.
\end{split}
\end{equation*}
This implies that
$$u_{\varepsilon,R}\geq0\quad\text{a.e.~in}~B_R.$$
On the other hand, we have
$$u^t=u_{\varepsilon,R},\quad\text{if}~u_{\varepsilon,R}\leq\sup\limits_D\phi,$$
$$u^t=(1-t)u_{\varepsilon,R}+t\sup_D\phi\geq\sup_D\phi,\quad\text{if}~u_{\varepsilon,R}\geq\sup\limits_D\phi.$$
By the above,
$$u^t\geq\phi\quad\text{a.e.~in}~D.$$
Together with $u^t\leq u_{\varepsilon,R}\leq\psi$ a.e. in $D$, we get $u^t\in K_R$. Moreover,
$$\mathcal{L}^n(\{x\in\Omega_R\mid u_{\varepsilon,R}(x)>0\})\geq\mathcal{L}^n(\{x\in\Omega_R\mid u^t(x)>0\}).$$
Again by the minimality of $u_{\varepsilon,R}$ and the first variation, we obtain
$$u_{\varepsilon,R}\leq\sup\limits_D\phi\quad\text{a.e.~in}~B_R.$$
This finishes the proof.
\end{proof}

In order to study the distribution of the positive phase, our next step is to derive the properties of $u_{\varepsilon,R}$ in the exterior domain $\Omega_{R}$.

\begin{lemma}\label{basic prop}
Let $u_{\varepsilon,R}$ be a solution to problem \eqref{eq:P-e-R}. Then $u_{\varepsilon,R}$ satisfies the following properties:
\begin{itemize}
\item [(1)] Lipschitz continuity:  $u_{\varepsilon,R}\in C^{0,1}_{\mathrm{loc}}(\Omega_R)$. Moreover, for any $\Omega'\Subset\Omega_R$, there exists a positive constant $C$ such that
$$||\nabla u_{\varepsilon,R}||_{L^{\infty}(\Omega')}\leq\frac{C}{\sqrt{\varepsilon}},$$
where $C$ depends only on $n$, $\mathrm{dist}(\Omega',\partial\Omega_R)$ and $\sup_D\phi$.
\item [(2)] Nondegeneracy: there exists a positive constant $C=C(n)$ such that for $B_r(x_0)\subset\Omega_{R}$, if
$$\frac{1}{r\mathcal{H}^{n-1}(\partial B_r(x_0))}\int_{\partial B_r(x_0)}u_{\varepsilon,R}\,\mathrm{d}\mathcal{H}^{n-1}\leq C\sqrt{\varepsilon},$$
then $u_{\varepsilon,R}\equiv0$ in $B_{\frac{r}{2}}(x_0)$.
\item [(3)] Positive density: for $\Omega'\Subset\Omega_R$, there exists a positive constant $c$ such that for $B_r(x_0)\subset\Omega'$ with $u_{\varepsilon,R}(x_0)>0$,
$$\frac{\mathcal{L}^n(B_r(x_0)\cap\{u_{\varepsilon,R}>0\})}{\mathcal{L}^n(B_r(x_0))}\geq c\varepsilon^n,$$
where $c$ depends only on $n$, $\mathrm{dist}(\Omega',\partial\Omega_R)$ and $\sup_D\phi$.
\item [(4)] Harmonicity: $u_{\varepsilon,R}$ is harmonic in $\{x\in\Omega_R\mid u_{\varepsilon,R}(x)>0\}$ and subharmonic in $\Omega_R$.
\item [(5)] Volume of positive phase: there exists a positive constant $M$ such that $$\mathcal{L}^n(\{x\in\Omega_{R}\mid u_{\varepsilon,R}(x)>0\})\leq\mu+M\varepsilon,$$
    where M depends only on $n$, $D$, $\mu$, $\phi$ and $\psi$.
\end{itemize}
\end{lemma}

\begin{proof}
Our proof starts with the observation that $u_{\varepsilon,R}$ restricted to $\Omega_R$ minimizes
$$J^{e}_{\varepsilon,R}(u):=\int_{\Omega_R}|\nabla u|^2\,\mathrm{d}x+f_{\varepsilon}(\mathcal{L}^n(\{x\in\Omega_R\mid u(x)>0\}))$$
over
$$K^e_{\varepsilon,R}:=\{u\in H^1(\Omega_R)\mid u\geq0~\text{a.e.~in}~\Omega_R,~u=u_{\varepsilon,R}~\text{on}~\partial D,~u=0~\text{on}~\partial B_R\}.$$
Here we used the fact that $u_{\varepsilon,R}\geq0$ a.e. in $B_R$ from Lemma \ref{lem:P-e-R,existence}. It is clear that
$$u_{\varepsilon,R}\geq\phi\quad\text{on}~\partial D$$
in the Sobolev trace sense. Combining $\phi>0$ on $\partial D$, $u_{\varepsilon,R}$ restricted to $\Omega_R$ actually solves the penalization problem considered in \cite{AAC 1986}. Moreover, since
$$\varepsilon(t_2-t_1)\leq f_{\varepsilon}(t_2-t_1)\leq\frac{1}{\varepsilon}(t_2-t_1)\quad\text{for}~t_1\leq t_2,$$
we follow almost without change as in \cite[Corollary 3.3 and Lemma 3.4]{AC 1981} to obtain (1) and (2).

Take $x_0$ and $B_r(x_0)$ as in (3). By (2), there is $y_0\in\partial B_{\frac{r}{2}}(x_0)$ such that
$$u_{\varepsilon,R}(y_0)>C(n)\sqrt{\varepsilon}r.$$
Combining the Lipschitz continuity (1), we obtain
$$u_{\varepsilon,R}(y)>0,\quad\text{if}~|y-y_0|\leq\frac{C(n)\sqrt{\varepsilon}r}{2\text{Lip}(u_{\varepsilon,R})}=:c\varepsilon,$$
and (3) is proved.

For any nonnegative $\eta\in C_0^{\infty}(\Omega_R)$, we have
$$0\leq\limsup_{t\rightarrow0^+}\frac{1}{2t}\left(J_{\varepsilon,R}(u_{\varepsilon,R}-t\eta)-J_{\varepsilon,R}(u_{\varepsilon,R})\right)
\leq-\int_{\Omega_R}\nabla u_{\varepsilon,R}\cdot\nabla\eta\,\mathrm{d}x.$$
That is, $u_{\varepsilon,R}$ is subharmonic in $\Omega_R$. By (1), $\{x\in\Omega_R\mid u_{\varepsilon,R}(x)>0\}$ is an open set. For any $\zeta\in C_0^{\infty}(\{x\in\Omega_R\mid u_{\varepsilon,R}(x)>0\})$ and $t>0$ small, we have
$$J_{\varepsilon,R}(u_{\varepsilon,R}\pm t\zeta)\geq J_{\varepsilon,R}(u_{\varepsilon,R}).$$
This implies
$$\int_{\{x\in\Omega_R\mid u_{\varepsilon,R}(x)>0\}}\nabla u_{\varepsilon,R}\cdot\nabla\zeta\,\mathrm{d}x=0,$$
and (4) is proved.

In the proof of Lemma \ref{lem:P-e-R,existence}, we see that
$$J_{\varepsilon,R}(u_{\varepsilon,R})\leq J_{\varepsilon,R}(w_0)=\int_{D'}|\nabla w_0|^2\,\mathrm{d}x=:M.$$
Thus,
$$f_{\varepsilon}(\mathcal{L}^n(\{x\in\Omega_R\mid u_{\varepsilon,R}(x)>0\}))\leq M.$$
By the definition of $f_{\varepsilon}$, the proof is finished.
\end{proof}

We continue by showing that the free boundary $\partial\{u_{\varepsilon,R}>0\}\cap\Omega_R$ stays away from the fixed boundary $\partial D$, i.e. $u_{\varepsilon,R}>0$ in a neighborhood of $D$.

Define for $0<\delta<1$,
$$D_{\delta}=\{x\in\mathbb{R}^n\mid\dist(x, D)<\delta\}.$$
Since $R\geq R_0$, together with \eqref{R_0}, we have $D_{\delta}\subset B_R$.

\begin{lemma}\label{D-nghd}
Let $u_{\varepsilon,R}$ be a solution to problem \eqref{eq:P-e-R}. Then there exists a positive constant $\delta$ such that
$$D_{\delta}\setminus\overline{D}\subset\{x\in\Omega_R\mid u_{\varepsilon,R}(x)>0\},$$
where $\delta$ depends only on $n$, $D$, $\varepsilon$, $\inf\limits_{\partial D}\phi$.
\end{lemma}
\begin{proof}
Recall that $u_{\varepsilon,R}$ restricted to $\Omega_R$ solves the penalization problem considered in \cite{AAC 1986}. In the proof of \cite[Lemma 6]{AAC 1986} (see also \cite[Lemma 3.4]{FMW 2006}), one has that for every $y_0\in\partial D$, there exists $\delta=\delta(n,D,\varepsilon,\inf\limits_{\partial D}\phi)>0$ such that
$$u_{\varepsilon,R}>0\quad\text{in}~B_{\delta}(y_0)\cap\Omega_R.$$
Using standard finite-cover argument, there exists a neighborhood of $\partial D$ in $\Omega_R$ where $u_{\varepsilon, R}>0$, and the conclusion follows.
\end{proof}

With Lemma \ref{basic prop} and Lemma \ref{D-nghd}, we finish this section by proving that solutions to problem $\eqref{eq:P-e-R}$ have uniformly bounded support independent of $R$.

\begin{prop}\label{uni bd}
Let $u_{\varepsilon,R}$ be a solution to problem $\eqref{eq:P-e-R}$. Then there exists a positive constant $M_0$ such that $\{x\in B_R\mid u_{\varepsilon,R}(x)>0\}$ is bounded by $B_{M_0}$. Here $M_0$ depends only on $n$, $\mu$, $\varepsilon$, $\delta$, $M$, $R_0$ and $\sup_D\phi$; $M$ and $\delta$ are given by Lemma \ref{basic prop} (5) and Lemma \ref{D-nghd}, respectively.
\end{prop}

\begin{proof}
Denote $W=\left\{x\in\Omega_R\mid u_{\varepsilon,R}(x)>0,~\dist(x,\partial \Omega_R)>\frac{\delta}{2}\right\}$. By Lemma \ref{D-nghd}, $W\neq\emptyset$.
By Lemma \ref{basic prop} (3), we have
$$\mathcal{L}^n\left(B_{\frac{\delta}{3}}(x)\cap\{u_{\varepsilon,R}>0\}\right)\geq c\varepsilon^n\delta^n,\quad\forall~x\in W.$$
By Vitali covering lemma, there exists a disjoint subcollection $$\left\{B_{\frac{\delta}{3}}(x_i)\right\}_{i\in\mathcal{I}}\subset\left\{B_{\frac{\delta}{3}}(x)\right\}_{x\in W}$$
such that $W\subset\bigcup_{i\in\mathcal{I}}B_{\frac{5}{3}\delta}(x_i)$, where $\mathcal{I}$ is an index set. It follows from Lemma \ref{basic prop} (5) that
\begin{equation*}
\begin{split}
\mu+M\varepsilon&\geq\mathcal{L}^n(\{x\in\Omega_R\mid u_{\varepsilon,R}(x)>0\})\\
&\geq\sum_{i\in\mathcal{I}}\mathcal{L}^n\left(B_{\frac{\delta}{3}}(x_i)\cap\{u_{\varepsilon,R}>0\}\right)\\
&\geq c\varepsilon^n\delta^n\mathrm{Card}(\mathcal{I}),
\end{split}
\end{equation*}
where $\mathrm{Card}(\mathcal{I})$ is the cardinality of $\mathcal{I}$. We thus get $\mathrm{Card}(\mathcal{I})$ is finite. Therefore, any point $x$ in $W$ can be connected to $D_{\delta}$ by a chain of at most $\mathrm{Card}(\mathcal{I})$ balls with radius $\frac{5}{3}\delta$. We finish the proof by taking
$$M_0=R_0+\delta+5\delta\mathrm{Card}(\mathcal{I}).$$
\end{proof}

\section{Solutions to problem \eqref{eq:P-e}}\label{sec:3}

In this section, we shall exploit qualitative results for the penalization problem \eqref{eq:P-e} that are analogous to those presented in Section \ref{sec:2}. In particular, we find that any solution $u_{\varepsilon}$ to problem \eqref{eq:P-e} has bounded support, which leads to the equivalence of problems \eqref{eq:P-e} and \eqref{eq:P-e-R} for large $R$. By illustrating that $u_\varepsilon$ grows linearly near the set $\partial\{u_\varepsilon>0\}\cap\Omega$, we also derive the regularity properties of the free boundary $\partial\{u_{\varepsilon}>0\}\cap\Omega$.

First, it is not hard to establish the existence of a solution to problem \eqref{eq:P-e} using the fact that solutions to problem \eqref{eq:P-e-R} have uniformly bounded support with respect to $R$ (Proposition
\ref{uni bd}).

\begin{lemma}\label{lem:P-e,existence}
There exists a solution to problem \eqref{eq:P-e}. Moreover, if $u_{\varepsilon}$ is such a solution, then
$$0\leq u_{\varepsilon}\leq\sup\limits_{D}\phi\quad\mathrm{a.e.~in}~\mathbb{R}^n.$$
\end{lemma}

\begin{proof}
We set
$$\omega=\inf\{J_{\varepsilon}(u)\mid u\in K\}.$$
Then $-\mu\leq\omega<\infty$ is finite. This follows by the same method as in Lemma \ref{lem:P-e-R,existence}. Let $\{u_k\}\subset K$ be a minimizing sequence. Without loss of generality, we may assume spt$(u_{k})\subset B_k$ for sufficiently large $k$. Moreover, we may also assume $u_k$ is a solution to problem $(P_{\varepsilon,k})$. Indeed, if $u_{\varepsilon,k}$ is a solution to problem $(P_{\varepsilon,k})$, by assuming $u_{\varepsilon,k}=0$ in $\mathbb{R}^n\setminus\overline{B}_k$, then
\begin{equation*}
\begin{split}
J_{\varepsilon}(u_{\varepsilon,k})&=\int_{\mathbb{R}^n}|\nabla u_{\varepsilon,k}|^2\,\mathrm{d}x+f_{\varepsilon}(\mathcal{L}^n(\{u_{\varepsilon,k}>0\}\cap\Omega))\\
&=\int_{B_k}|\nabla u_{\varepsilon,k}|^2\,\mathrm{d}x+f_{\varepsilon}(\mathcal{L}^n(\{u_{\varepsilon,k}>0\}\cap\Omega_k))\\
&=J_{\varepsilon,k}(u_{\varepsilon,k})\leq J_{\varepsilon,k}(u_k)=J_{\varepsilon}(u_k).
\end{split}
\end{equation*}
By Proposition \ref{uni bd}, we have $$\text{spt}(u_k)\subset B_{M_0}.$$
Fix some sufficiently large $k_0>M_0$ and take a corresponding solution $u_{\varepsilon}$ to problem $(P_{\varepsilon,k_0})$. Assume that $u_\varepsilon=0$ in $\mathbb{R}^n\setminus\overline{B}_{k_0}$. We thus get
$$J_{\varepsilon}(u_{\varepsilon})=J_{\varepsilon,k_0}(u_\varepsilon)\leq J_{\varepsilon,k_0}(u_k)\leq J_{\varepsilon}(u_k),\quad\forall~k\geq1.$$
Letting $k\rightarrow\infty$, we obtain that $u_{\varepsilon}$ is a solution to problem \eqref{eq:P-e}.

Analysis similar to that in the proof of Lemma \ref{lem:P-e-R,existence} shows the $\mathcal{L}^\infty$-estimate on $u_{\varepsilon}$ in $\mathbb{R}^n$.
This finishes the proof.
\end{proof}

Proceeding as in Section \ref{sec:2} with minor modifications, we can easily obtain certain fine properties for solutions to problem \eqref{eq:P-e}. We state them below for complements.

\begin{lemma}\label{prop-2}
Let $u_{\varepsilon}$ be a solution to problem \eqref{eq:P-e}. Then $u_{\varepsilon}$ satisfies the following properties:
\begin{enumerate}
\item [(1)] $u_{\varepsilon}\in C_{\mathrm{loc}}^{1,1}(D)\cap C_{\mathrm{loc}}^{0,1}(\Omega)$. Moreover, for any $\Omega'\Subset\Omega$, there exists a positive constant $C$ such that
$$||\nabla u_{\varepsilon}||_{L^{\infty}(\Omega')}\leq\frac{C}{\sqrt{\varepsilon}},$$
where $C$ depends only on $n$, $\mathrm{dist}(\Omega',\partial\Omega)$ and $\sup\limits_D\phi$.
\item [(2)] There exists a constant $C=C(n)$ such that for $B_r(x_0)\subset\Omega$, if
$$\frac{1}{r\mathcal{H}^{n-1}(\partial B_r(x_0))}\int_{\partial B_r(x_0)}u_{\varepsilon}\,\mathrm{d}\mathcal{H}^{n-1}\leq C\sqrt{\varepsilon},$$
then $u_{\varepsilon}\equiv0$ in $B_{\frac{r}{2}}(x_0)$.
\item [(3)] For $\Omega'\Subset\Omega$, there exists a positive constant $c$ such that for $B_r(x_0)\subset\Omega'$ with $u_{\varepsilon}(x_0)>0$,
$$\frac{\mathcal{L}^n(B_r(x_0)\cap\{u_{\varepsilon}>0\})}{\mathcal{L}^n(B_r(x_0))}\geq c\varepsilon^n,$$
where $c$ depends only on $n$, $\mathrm{dist}(\Omega',\partial\Omega)$ and $\sup_D\phi$.
\item [(4)] $u_{\varepsilon}$ is harmonic in $\{x\in\Omega\mid u_{\varepsilon}(x)>0\}$ and subharmonic in $\Omega$.
\item [(5)] $\mathcal{L}^n(\{x\in\Omega\mid u_{\varepsilon}(x)>0\})\leq\mu+M\varepsilon$, where
$M$ is given by Lemma \ref{basic prop} (5).
\end{enumerate}
\end{lemma}

\begin{proof}
It is easy to check that $u_{\varepsilon}$ restricted to $D$ solves the double obstacle problem in $D$ with $\phi$, $\psi$ as obstacles and boundary data $u_{\varepsilon}$ on $\partial D$. Namely, $u_\varepsilon$ minimizes the energy functional $\int_D|\nabla u|^2\,\mathrm{d}x$
over
$$\{u\in H^1(D)\mid\phi\leq u\leq\psi~\text{a.e.~in}~D,~u=u_\varepsilon~\text{on}~\partial D\}.$$
The optimal regularity theorem in \cite{CK 1980,C 1979} implies that $u_{\varepsilon}\in C_{\mathrm{loc}}^{1,1}(D)$.

Also, using the fact from Lemma \ref{lem:P-e,existence} that $u_{\varepsilon}\geq0$ a.e. in $\Omega$ again, $u_{\varepsilon}$ also minimizes
$$J^e_{\varepsilon}(u):=\int_{\Omega}|\nabla u|^2\,\mathrm{d}x+f_{\varepsilon}(\mathcal{L}^n(\{x\in\Omega\mid u(x)>0\}))$$
over
$$K^e_{\varepsilon}:=\{u\in H^1(\Omega)\mid u\geq0~\text{a.e.~in}~\Omega,~u=u_{\varepsilon}~\text{on}~\partial\Omega\}.$$
Consequently, $u_{\varepsilon}$ solves the penalization problem in $\Omega$ studied in \cite{AAC 1986}. The rest of the proof runs as in Lemma \ref{basic prop}.
\end{proof}

Likewise, one can see that the following assertions hold.

\begin{lemma}\label{D-nghd2}
Let $u_\varepsilon$ be a solution to problem \eqref{eq:P-e}. Then
$$D_{\delta}\setminus\overline{D}\subset\{x\in\Omega\mid u_{\varepsilon}(x)>0\},$$
where $\delta$ is given by Lemma \ref{D-nghd}.
\end{lemma}

By virtue of the above lemmas, we can argue as in the proof of Proposition \ref{uni bd} to verify that any solution to problem \eqref{eq:P-e} has bounded support. As an immediate consequence, we get the equivalence of problems \eqref{eq:P-e} and \eqref{eq:P-e-R} for large $R$.

\begin{prop}\label{prop:u-e bdd spt}
Let $u_\varepsilon$ be a solution to problem \eqref{eq:P-e}. Then $\{x\in\mathbb{R}^n\mid u_{\varepsilon}(x)>0\}$ is bounded by $B_{M_0}$, where $M_0$ is given by Proposition \ref{uni bd}.
\end{prop}

\begin{corollary}
For every $R>M_0$, $u_{\varepsilon}$ is a solution to problem \eqref{eq:P-e} if and only if $u_{\varepsilon}$ is a solution to problem \eqref{eq:P-e-R}.
\end{corollary}

Recall that the solution $u_\varepsilon$ to problem \eqref{eq:P-e} is nonnegative in $\Omega$ and harmonic in $\{u_\varepsilon>0\}\cap\Omega$, and with linear growth near the free boundary $\partial\{u_\varepsilon>0\}\cap\Omega$. Here the linear growth comes from Lipschitz regularity and nondegeneracy (Lemma \ref{prop-2} (1) and (2)). In addition, note that $u_\varepsilon$ restricted to $\Omega$ solves the penalization problem studied in \cite{AAC 1986}. Therefore, we are allowed to apply the corresponding results in \cite{AAC 1986,AC 1981} to $u_\varepsilon$ to obtain the representation theorem, which contains information on the regularity of the free boundary.  To sum up, we obtain the following properties of the free boundary $\partial\{u_\varepsilon>0\}\cap\Omega$.

\begin{theorem}\label{properties}
Let $u_\varepsilon$ be a solution to problem \eqref{eq:P-e}. Then
\begin{enumerate}
  \item [(1)] $\mathcal{H}^{n-1}(\partial\{u_{\varepsilon}>0\}\cap\Omega')<\infty$ for $\Omega'\Subset\Omega$.
  \item [(2)] There exists a Borel function $q_{u_{\varepsilon}}$ such that in the sense of distributions
$$\Delta u_{\varepsilon}=q_{u_{\varepsilon}}\mathcal{H}^{n-1}\lfloor\partial\{u_{\varepsilon}>0\}.$$
That is, for $\eta\in C_0^{\infty}(\Omega)$, there holds
$$-\int_{\Omega}\nabla u_{\varepsilon}\cdot\nabla\eta\,\mathrm{d}x=\int_{\partial\{u_{\varepsilon}>0\}\cap\Omega} q_{u_{\varepsilon}}\eta\,\mathrm{d}\mathcal{H}^{n-1}.$$
Moreover, there exists a positive constant $\lambda_{\varepsilon}$ such that
$$q_{u_{\varepsilon}}=\lambda_{\varepsilon}\quad\mathcal{H}^{n-1}\text{-}\mathrm{a.e.~on}~\partial\{u_{\varepsilon}>0\}\cap\Omega.$$
  \item [(3)] For $\Omega'\Subset\Omega$, there exist positive constants $c$ and $C$ such that for $B_r(x_0)\Subset\Omega'$ with $x_0\in\partial\{u_\varepsilon>0\}$,
  $$cr^{n-1}\leq\mathcal{H}^{n-1}(\partial\{u_\varepsilon>0\}\cap B_r(x_0))\leq Cr^{n-1},$$
  where $c$ and $C$ depend only on $n$, $u_\varepsilon$, $\Omega'$ and $\Omega$.
  \item [(4)] For $\mathcal{H}^{n-1}$-almost every $x_0\in\partial_{\mathrm{red}}\{u_\varepsilon>0\}$, there holds
  $$u_\varepsilon(x_0+x)=q_{u_\varepsilon}(x_0)\max\{-x\cdot\nu(x_0),0\}+o(|x|),\quad\mathrm{as}~|x|\rightarrow0,$$
  where $\nu(x_0)$ is the outer normal vector to the reduced boundary of $\{u_\varepsilon>0\}$ at $x_0$.
  \item [(5)] $\partial\{u_{\varepsilon}>0\}\cap\Omega$ is smooth, except for a closed set of $\mathcal{H}^{n-1}$-measure zero.
\end{enumerate}
\end{theorem}

\begin{proof}
The proof follows exactly as that of \cite[\S4]{AC 1981} and \cite[\S2]{AAC 1986}.
\end{proof}

\section{Behavior of solutions to problem \eqref{eq:P-e} for small $\varepsilon$}\label{sec:4}
This section aims to reveal the fact that a solution $u_{\varepsilon}$ to problem \eqref{eq:P-e} also solves problem \eqref{eq:P} when the penalization parameter $\varepsilon$ is small enough. Actually, it is only needed to verify that the volume of $\{u_{\varepsilon}>0\}\setminus\overline{D}$ adjusts to the prescribed volume $\mu$ for small $\varepsilon$. Here, unlike Theorem \ref{properties}, we cannot apply the perturbation result in \cite{AAC 1986} directly, since our boundary data on $\partial D$ vary with $\varepsilon$. Instead, we handle it by combining the ideas in \cite{AAC 1986} with a comparison argument.

\begin{lemma}\label{upp bd}
Let $u_\varepsilon$ be a solution to problem \eqref{eq:P-e}. Then
$$c\leq\lambda_{\varepsilon}\leq C,$$
where $\lambda_\varepsilon$ is given by Theorem \ref{properties} (2), $c$ and $C$ are positive constants independent of $\varepsilon$.
\end{lemma}

\begin{proof}
We first prove that there are positive constants $c$ and $C$, independent of $\varepsilon$, such that
$$c\leq\mathcal{L}^n(\{x\in\Omega_{R_0}\mid u_{\varepsilon}(x)>0\})\leq \mu+C\varepsilon,$$
where $R_0$ is given in \eqref{R_0}. The estimate from above is already proved in Lemma \ref{prop-2} (5). On the other hand, in the proof of Lemma \ref{lem:P-e-R,existence} and Lemma \ref{lem:P-e,existence}, we see that
$$\int_{\Omega}|\nabla u_{\varepsilon}|^2\,\mathrm{d}x\leq C.$$
By Proposition \ref{prop:u-e bdd spt}, $u_\varepsilon$ has bounded support. Then Poincar\'{e} inequality gives
$$||u_{\varepsilon}||_{H^1(\Omega)}\leq C.$$
Note that $u_\varepsilon\geq0$ in $\Omega_{R_0}$ and $u_{\varepsilon}\geq\phi$ on $\partial D$ in the Sobolev trace sense. By Sobolev trace theorem and H\"{o}lder inequality, we have
\begin{equation*}
\begin{split}
\int_{\partial D}\phi\,\mathrm{d}\mathcal{H}^{n-1}&\leq\int_{\partial\Omega_{R_0}}u_{\varepsilon}\,\mathrm{d}\mathcal{H}^{n-1}\\
&\leq C\mathcal{L}^n(\{x\in\Omega_{R_0}\mid u_{\varepsilon}(x)>0\})^{\frac{1}{2}}||u_{\varepsilon}||_{H^1(\Omega_{R_0})}\\
&\leq C\mathcal{L}^n(\{x\in\Omega_{R_0}\mid u_{\varepsilon}(x)>0\})^{\frac{1}{2}}.
\end{split}
\end{equation*}
We thus obtain the estimate from below.

Take a smooth domain $\Omega'\Subset \Omega_{R_0}$ with
$$\mathcal{L}^n(\Omega')>\mu\quad\text{and}\quad\mathcal{L}^n(\Omega_{R_0}\setminus\Omega')<c,$$
where $c$ is the constant as above. Then for $\varepsilon$ small enough,
$$\mathcal{L}^n(\Omega'\cap\{u_{\varepsilon}>0\})\leq \mu+C\varepsilon<\mathcal{L}^n(\Omega').$$
On the other hand,
$$\mathcal{L}^n(\Omega'\cap\{u_{\varepsilon}>0\})\geq\mathcal{L}^n(\Omega_{R_0}\cap\{u_{\varepsilon}>0\})-\mathcal{L}^n(\Omega_{R_0}\setminus\Omega')\geq c-\mathcal{L}^n(\Omega_{R_0}\setminus\Omega')>0.$$
Hence, by the relative isoperimetric inequality, we have
$$\mathcal{H}^{n-1}(\Omega'\cap\partial\{u_{\varepsilon}>0\})\geq c\min\{\mathcal{L}^n(\Omega'\cap\{u_{\varepsilon}>0\}),\mathcal{L}^n(\Omega'\cap\{u_{\varepsilon}=0\})\}^{\frac{n-1}{n}}\geq c>0.$$
Let $h_{\varepsilon}$, $h$ be the harmonic functions in $\Omega_{R_0}$ vanishing on $\partial B_{R_0}$ and equal to $u_{\varepsilon}$, $\phi$ on $\partial D$, respectively. By the comparison principle and strong maximum principle,
$$h_{\varepsilon}\geq h>0\quad\text{in}~\Omega_{R_0}.$$
Assume $h_\varepsilon=0$ in $\mathbb{R}^n\setminus\overline{B}_{R_0}$. Since $\int_{\Omega}|\nabla u_{\varepsilon}|^2\,\mathrm{d}x$ is bounded, it follows from Theorem \ref{properties} (2) that
\begin{equation*}
\begin{split}
C&\geq\int_{\Omega}\nabla u_{\varepsilon}\cdot\nabla\min\{u_{\varepsilon}-h_{\varepsilon},0\}\,\mathrm{d}x\\
&=\int_{\partial\{u_{\varepsilon}>0\}\cap\Omega}\lambda_{\varepsilon}\max\{h_\varepsilon-u_\varepsilon,0\}\,\mathrm{d}\mathcal{H}^{n-1}\\
&=\int_{\partial\{u_{\varepsilon}>0\}\cap\Omega_{R_0}}\lambda_{\varepsilon}h_{\varepsilon}\,\mathrm{d}\mathcal{H}^{n-1}\\
&\geq\lambda_{\varepsilon}\inf_{\Omega'}h\,\mathcal{H}^{n-1}(\partial\{u_{\varepsilon}>0\}\cap\Omega')\geq c\lambda_{\varepsilon}.
\end{split}
\end{equation*}
This gives an upper bound for $\lambda_\varepsilon$.

Again by the fact that $u_{\varepsilon}\geq\phi$ on $\partial D$ in the Sobolev trace sense and $\phi$ is strictly positive on $\partial D$, the same arguments as in \cite[Lemma 6]{AAC 1986} can be applied to $u_\varepsilon$ to give the lower bound for $\lambda_\varepsilon$. This finishes the proof.
\end{proof}

With the uniform bounds on $\lambda_{\varepsilon}$ and the fact that $f'_\varepsilon(s)$ jumps at $s=\mu$, we are ready to prove the desired result.

\begin{prop}\label{small epsi}
There exists $\varepsilon_0>0$ such that if $u_\varepsilon$ is a solution to problem \eqref{eq:P-e}, then for $0<\varepsilon<\varepsilon_0$,
$$\mathcal{L}^n(\Omega\cap\{u_{\varepsilon}>0\})=\mu.$$
Therefore, $u_{\varepsilon}$ is a solution to problem \eqref{eq:P} for such $\varepsilon$. On the contrary, if $u$ is a solution to problem \eqref{eq:P}, then $u$ is a solution to problem \eqref{eq:P-e} for any $0<\varepsilon<\varepsilon_0$.
\end{prop}
\begin{proof}
We prove it by contradiction. Suppose first that $\mathcal{L}^n(\Omega\cap\{u_{\varepsilon}>0\})>\mu$. Let $x_0\in\partial\{u_{\varepsilon}>0\}$ be a regular free boundary point, that is, $\partial\{u_{\varepsilon}>0\}$ is smooth in $B_r(x_0)$ for small $r>0$. Near $x_0$ we make a smooth inward perturbation of the set $\{u_{\varepsilon}>0\}$ decreasing its volume by $\delta(V)$. We choose $\delta(V)$ small such that
$$\mathcal{L}^n(\Omega\cap\{u_{\varepsilon}>0\})-\delta(V)\geq\mu.$$
Denote the perturbed set by $P$, where
$$P\subset B_r(x_0)\cap\{u_{\varepsilon}>0\}.$$
Let $v_{\varepsilon}$ be the function in $B_r(x_0)$ satisfying
\begin{equation*}
\begin{cases}
\Delta v_{\varepsilon}=0,\quad&\text{in}~P,\\
v_{\varepsilon}=0,\quad&\text{in}~B_r(x_0)\setminus P,\\
v_{\varepsilon}=u_{\varepsilon},\quad&\text{on}~\partial P\cap\partial B_{r}(x_0),
\end{cases}
\end{equation*}
and let
$$
w_{\varepsilon}=
\begin{cases}
v_{\varepsilon},\quad&\text{in}~B_r(x_0),\\
u_{\varepsilon},\quad&\text{in}~\mathbb{R}^n\setminus B_r(x_0).
\end{cases}
$$
By Hadamard variational principle and Lemma \ref{upp bd}, we have
\begin{equation*}
\begin{split}
\int_{B_r(x_0)}(|\nabla v_{\varepsilon}|^2-|\nabla u_{\epsilon}|^2)\,\mathrm{d}x&=\lambda_{\varepsilon}^2\delta(V)+o(\delta(V))\\
&\leq C^2\delta(V)+o(\delta(V)).
\end{split}
\end{equation*}
Hence,
\begin{equation*}
\begin{split}
0\leq&J_{\varepsilon}(w_{\varepsilon})-J_{\varepsilon}(u_{\varepsilon})\\
=&\int_{B_r(x_0)}|\nabla v_{\varepsilon}|^2\,\mathrm{d}x-\int_{B_r(x_0)}|\nabla u_{\varepsilon}|^2\,\mathrm{d}x\\
&+f_{\varepsilon}(\mathcal{L}^n(\Omega\cap\{w_\varepsilon>0\}))-f_{\varepsilon}(\mathcal{L}^n(\Omega\cap\{u_\varepsilon>0\}))\\
\leq&C^2\delta(V)-\frac{1}{\varepsilon}\delta(V)+o(\delta(V))<0,
\end{split}
\end{equation*}
provided $\varepsilon<\frac{1}{2C^2}$, which is a contradiction.

Now suppose that $\mathcal{L}^n(\Omega\cap\{u_{\varepsilon}>0\})<\mu$. We make a smooth outward perturbation near a regular free boundary point. Due the facts that $\lambda_\varepsilon$ has a lower bound and
$$f'_\varepsilon(s)=\varepsilon\quad\text{for}~s<\mu,$$
similar arguments leads to a contradiction. Thus, we conclude that if $\varepsilon$ is small enough,
$$\mathcal{L}^n(\Omega\cap\{u_{\varepsilon}>0\})=\mu,$$
and the conclusion follows immediately.
\end{proof}

\section{Proofs of Theorem \ref{thm:main 1} and Theorem \ref{thm:main 2}}\label{sec:5}

At this stage, since the original problem \eqref{eq:P} is recovered from problem \eqref{eq:P-e} for small $\varepsilon$ (Proposition \ref{small epsi}), we already know that problem \eqref{eq:P} admits a solution $u_{\varepsilon}$ which is in $C^{0,1}_{\mathrm{loc}}(\mathbb{R}^n\setminus\overline{D})\cap C^{1,1}_{\mathrm{loc}}(D)$ (Lemmas \ref{lem:P-e,existence} and \ref{prop-2}), but the regularity up to $\partial D$ has not been treated. To fill this gap, we will enlarge the obstacles to be piecewise smooth in a neighborhood of $D$ and identify $u_{\varepsilon}$ as a solution to a double obstacle problem near $\partial D$. Then by reducing that problem locally to a single obstacle problem, we prove that $u_{\varepsilon}$ is Lipschitz continuous across $\partial D$. Consequently, together with Theorem \ref{properties}, we can conclude Theorem \ref{thm:main 1}. Furthermore, we prove that $u_\varepsilon$ is $C^{1,\frac{1}{2}}$ up to $\partial D$ and thus obtain Theorem \ref{thm:main 2}.

To begin with, take $\delta_\varepsilon=\frac{\delta}{2}$ where $\delta$ is as in Lemma \ref{D-nghd2}. Let us consider the domain $D_{\delta_\varepsilon}\Supset D$ and enlarge the obstacle $\phi$ by taking an auxiliary function $\phi_{\varepsilon}$ such that
\begin{equation*}
\begin{cases}
\Delta \phi_{\varepsilon}=0,\quad&\text{in}~D_{\delta_\varepsilon}\setminus\overline{D},\\
\phi_{\varepsilon}=\phi,\quad&\text{on}~\overline{D},\\
\phi_{\varepsilon}=u_{\varepsilon},~\quad&\text{on}~\partial D_{\delta_\varepsilon}.
\end{cases}
\end{equation*}
Enlarge $\psi$ in the same way and denote the resulting function by $\psi_{\varepsilon}$. Clearly, $u_\varepsilon$ is smooth on $\partial D_{\delta_\varepsilon}$ in view of Lemma \ref{prop-2} (4) and Lemma \ref{D-nghd2}. By boundary estimates for harmonic functions, it is seen that
$$\phi_{\varepsilon},~\psi_{\varepsilon}\in C^{0,1}(\overline{D}_{\delta_{\varepsilon}})\cap C^{1,1}(\overline{D})\cap C^{1,1}(\overline{D}_{\delta_{\varepsilon}}\setminus D).$$
Let $\varepsilon_0$ be as in Proposition \ref{small epsi}. Given $0<\varepsilon<\varepsilon_0$, we show that $u_{\varepsilon}$ solves the double obstacle problem in $D_{\delta_\varepsilon}$ with $\phi_{\varepsilon}$ and $\psi_{\varepsilon}$ as lower and upper obstacles, respectively.

\begin{lemma}\label{u-epsi}
Let $u_{\varepsilon}$ be a solution to problem \eqref{eq:P-e}. Then for $0<\varepsilon<\varepsilon_0$, $u_{\varepsilon}$ minimizes the energy functional
$$J^i_{\varepsilon}(u):=\int_{D_{\delta_\varepsilon}}|\nabla u|^2\,\mathrm{d}x$$
over
$$K^i_{\varepsilon}:=\{u\in H^1(D_{\delta_\varepsilon})\mid\phi_{\varepsilon}\leq u\leq\psi_{\varepsilon}~\mathrm{a.e.~in}~D_{\delta_\varepsilon},~u=u_{\varepsilon}~\mathrm{on}~\partial D_{\delta_\varepsilon}\}.$$
\end{lemma}

\begin{proof}
By Lemma \ref{prop-2} (4) and Lemma \ref{D-nghd2}, we have
\begin{equation}\label{eq:Delta-u-e=0}
\Delta u_{\varepsilon}=0\quad\text{in}~D_{\delta_\varepsilon}\setminus\overline{D}.
\end{equation}
By the definition of $\phi_{\varepsilon}$, $\psi_{\varepsilon}$ and $\phi\leq u_{\varepsilon}\leq\psi$ a.e. in $D$, the comparison principle implies
$$\phi_{\varepsilon}\leq u_{\varepsilon}\leq\psi_{\varepsilon}\quad\text{a.e.~in}~D_{\delta_\varepsilon}.$$
Hence $u_{\varepsilon}\in K^i_{\varepsilon}$.

Since $\phi_{\varepsilon}>0$ on $\partial(D_{\delta_\varepsilon}\setminus\overline{D})$, the strong maximum principle gives the positivity of $\phi_{\varepsilon}$ in $D_{\delta_\varepsilon}$. We thus get
$$u>0\quad\text{a.e.~in}~D_{\delta_\varepsilon},\quad\forall~u\in K^i_{\varepsilon}.$$
By assuming $u=u_{\varepsilon}$ in $\mathbb{R}^n\setminus\overline{D}_{\delta_{\varepsilon}}$, the function $u\in K^i_{\varepsilon}$ can be considered as belonging to $K$. It follows from Proposition \ref{small epsi} that
$$\mathcal{L}^n(\{u>0\}\setminus\overline{D})=\mathcal{L}^n(\{u_{\varepsilon}>0\}\setminus\overline{D})=\mu,\quad\forall~u\in K^i_{\varepsilon}.$$
Using the minimality of $u_{\varepsilon}$, the conclusion follows.
\end{proof}

Based on the above observation, we can further derive the Lipschitz continuity of $u_{\varepsilon}$ in $D_{\delta_\varepsilon}$. For this,
we need to introduce the following estimate for harmonic functions.
\begin{lemma}\label{control}
Let $U$ be a bounded open set in $\mathbb{R}^n$. Let $h\in C^0(\overline{U\cap B_r})$ be harmonic and $g\in W^{1,\infty}(U\cap B_r)\cap C^0(\overline{U\cap B_r})$ with $\|g\|_{ W^{1,\infty}(U\cap B_r)}\leq L$. Assume that
$$|h(x)-g(x)|\leq L\,\mathrm{dist}(x,\partial U\cap B_r),\quad\forall~x\in U\cap B_r.$$
Then there exists a positive constant $C=C(n)$ such that
$$|h(x)-h(y)|\leq CL|x-y|$$
for any $x,y\in U\cap B_r$ satisfying
\begin{equation}\label{U-nbhd}
\mathrm{dist}(x,\partial U)<\frac{1}{4}(r-|x|)\quad\mathrm{and}\quad\mathrm{dist}(y,\partial U)<\frac{1}{4}(r-|y|).
\end{equation}

\end{lemma}

\begin{proof}
We consider two cases:
\begin{enumerate}
  \item [(1)] Suppose that $|x-y|\geq\frac{1}{4}\max\{\dist(x,\partial U),\dist(y,\partial U)\}$. Take $x^*,y^*\in\partial U$ such that
$$|x-x^*|=\dist(x,\partial U)\quad\text{and}\quad|y-y^*|=\dist(y,\partial U).$$
By assumption \eqref{U-nbhd}, we have $x^*,y^*\in\partial U\cap B_r$. Then
\begin{equation*}
\begin{split}
|h(x)-h(y)|&\leq|h(x)-g(x)|+|g(x)-g(y)|+|h(y)-g(y)|\\
&\leq L|x-x^*|+L|x-y|+L|y-y^*|\\
&\leq 9L|x-y|.
\end{split}
\end{equation*}
  \item [(2)] Suppose that $|x-y|<\frac{1}{4}\dist(x,\partial U)=\frac{1}{4}\max\{\dist(x,\partial U),\dist(y,\partial U)\}$. Assume $x^*$ as above and denote $\rho=\dist(x,\partial U)$. Then $y\in B_{\frac{\rho}{4}}(x)$ and $B_{\frac{\rho}{2}}(x)\subset U\cap B_{r}$ by \eqref{U-nbhd} again.
Since $h$ is harmonic in $B_{\frac{\rho}{2}}(x)$, the interior gradient estimate gives
\begin{equation*}
\begin{split}
\|\nabla h\|_{L^{\infty}\left(B_{\frac{\rho}{4}(x)}\right)}&\leq\frac{C}{\rho}\|h-g(x^*)\|_{L^{\infty}\left(B_{\frac{\rho}{2}}(x)\right)}\\
&\leq\frac{C}{\rho}\left(\|h-g\|_{L^{\infty}\left(B_{\frac{\rho}{2}}(x)\right)}+\|g-g(x^*)\|_{L^{\infty}\left(B_{\frac{\rho}{2}}(x)\right)}\right)\\
&\leq\frac{2C}{\rho}(L\rho+L\rho)=2CL.
\end{split}
\end{equation*}
Hence, $|h(x)-h(y)|\leq 2CL|x-y|$.
\end{enumerate}
\end{proof}

By approximation arguments, we are able to apply regularity theory in classical obstacle problems to deduce the regularity of $u_\varepsilon$ in $D_{\delta_\varepsilon}$.

\begin{lemma}\label{Lip}
Let $u_\varepsilon$ be a solution to problem \eqref{eq:P-e}. Then for $0<\varepsilon<\varepsilon_0$,
$$u_{\varepsilon}\in C^{0,1}(\overline{D}_{\delta_{\varepsilon}}).$$
\end{lemma}

\begin{proof}
In the proof, we would often omit the subscript $\varepsilon$ in $u_{\varepsilon}$, $\phi_\varepsilon$, $\psi_\varepsilon$ and $D_{\delta_\varepsilon}$ for simplicity. Define for $0<\sigma<1$,
$$\phi^{\sigma}=\phi-\sigma,\quad\psi^{\sigma}=\psi+\sigma\quad\text{in}~D_\delta,$$
and
$$K^{\sigma}=\{v\in H^1(D_{\delta})\mid \phi^{\sigma}\leq v\leq\psi^{\sigma}~\text{a.e.~in}~D_{\delta},~v=u~\text{on}~\partial D_{\delta}\}.$$
Denote $v^{\sigma}$ the minimizer of the energy functionl $\int_{D_{\delta}}|\nabla v|^2\,\mathrm{d}x$ over $K^{\sigma}$. By Lemma \ref{u-epsi},  we can apply the pointwise regularity result of double obstacle problem in \cite{DMV 1989} to obtain $v^{\sigma}\in C^0(D_{\delta})$, and so the coincidence sets
$$\Lambda^{\sigma}_1:=\{x\in D_{\delta}\mid v^{\sigma}(x)=\phi^{\sigma}(x)\}\quad\text{and}
\quad\Lambda^{\sigma}_2:=\{x\in D_{\delta}\mid v^{\sigma}(x)=\psi^{\sigma}(x)\}$$
are compact and disjoint. It is easy to check
\begin{equation}\label{eq:v-sigma-harm}
\Delta v^\sigma=0\quad\text{in}~D_\delta\setminus(\Lambda^{\sigma}_1\cup\Lambda^{\sigma}_2).
\end{equation}
We will take a smooth open set $V_1$ such that
$$\Lambda^{\sigma}_1\Subset V_1\Subset D_{\delta}\setminus\Lambda_2^\sigma.$$
This implies
$$\int_{V_1}|\nabla v^\sigma|^2\,\mathrm{d}x=\min\left\{\int_{V_1}|\nabla v|^2\,\mathrm{d}x\mid v\in H^1(V_1),~v\geq\phi^\sigma~\text{a.e.~in}~V_1,~v=v^\sigma~\text{on}~\partial V_1\right\}.$$
That is, $v^\sigma$ restricted to $V_1$ solves the single obstacle problem with $\phi^\sigma$ as obstacle.

We claim that there exists a smooth neighborhood $U_1\subset V_1$ of $\Lambda_1^{\sigma}$ such that
\begin{equation}\label{est-1}
|\nabla v^{\sigma}|\leq C\|\nabla\phi\|_{L^{\infty}(D_{\delta})}\quad\text{a.e.}~\text{in}~U_1,
\end{equation}
where $C=C(n)$ is a positive constant. Indeed, it is sufficient to show that
\begin{equation}\label{double-gra-est}
|v^{\sigma}(x)-v^{\sigma}(y)|\leq C\|\nabla\phi\|_{L^{\infty}(D_{\delta})}|x-y|
\end{equation}
for any $x,y\in V_1$ satisfying
\begin{equation*}
\dist(x,\Lambda_1^{\sigma})<\frac{1}{4}(r_0-|x-x^*|)\quad\text{and}\quad\dist(y,\Lambda_1^{\sigma})<\frac{1}{4}(r_0-|y-x^*|)
\end{equation*}
for some $x^*\in\Lambda_1^{\sigma}$, where $r_0=\frac{1}{2}\dist(\Lambda_1^{\sigma},\partial V_1)$.
To show \eqref{double-gra-est}, we consider three cases:
\begin{enumerate}
\item [(1)] Suppose that $x,y\in\Lambda_1^{\sigma}$. Then
$$|v^{\sigma}(x)-v^{\sigma}(y)|=|\phi(x)-\phi(y)|\leq\|\nabla\phi\|_{L^{\infty}(D_{\delta})}|x-y|.$$
\item [(2)] Suppose that $x\in V_1\setminus\Lambda_{1}^{\sigma}$, $y\in\Lambda_1^{\sigma}$. Take $x^*\in\partial\Lambda_1^{\sigma}$ such that
$\dist(x,\partial\Lambda_1^{\sigma})=|x-x^*|$. It follows from (1) that
\begin{equation*}
\begin{split}
|v^{\sigma}(x)-v^{\sigma}(y)|&\leq|v^{\sigma}(x)-v^{\sigma}(x^*)|+|v^{\sigma}(x^*)-v^{\sigma}(y)|\\
&\leq C\|\nabla\phi\|_{L^{\infty}(D_{\delta})}|x-x^*|+C\|\nabla\phi\|_{L^{\infty}(D_{\delta})}|x^*-y|\\
&\leq C\|\nabla\phi\|_{L^{\infty}(D_{\delta})}|x-y|,
\end{split}
\end{equation*}
where we also used the fact that
$$|v^{\sigma}(x)-v^{\sigma}(x^*)|\leq C(n)\|\nabla\phi\|_{L^{\infty}(V_1)}|x-x^{*}|$$
in the second inequality, which is proved in single obstacle problems (see for instance \cite[Theorem 3.1]{LFH 2016}).
\item [(3)] Suppose that $x,y\in V_1\setminus\Lambda_{1}^{\sigma}$. Using the standard estimates in single obstacle problems (see \cite[Lemma 2]{C 1988}), we have
$$\sup_{B_{\frac{r}{2}}(x^*)}|v^{\sigma}(x)-\phi^{\sigma}(x)|\leq C(n)\|\nabla\phi\|_{L^{\infty}(D_{\delta})}r,$$
for any $x^*\in\Lambda_1^{\sigma}$ and $r<r_0$. By \eqref{eq:v-sigma-harm}, we can apply Lemma \ref{control} to $v^{\sigma}$ and obtain
$$|v^{\sigma}(x)-v^{\sigma}(y)|\leq C\|\nabla\phi\|_{L^{\infty}(D_{\delta})}|x-y|.$$
\end{enumerate}
The claim is proved.

To continue we apply similar arguments to $-v^{\sigma}$ and $-\psi^{\sigma}$ and find a smooth neighborhood $U_2$ of $\Lambda^{\sigma}_2$ such that
\begin{equation}\label{est-2}
|\nabla v^{\sigma}|\leq C||\nabla\psi||_{L^{\infty}(D_{\delta})}\quad\text{a.e.~in}~U_2.
\end{equation}
On the other hand, note that $v^{\sigma}$ is harmonic for some neighborhood of $\partial D_{\delta}$ in $D_{\delta}$ due to \eqref{eq:v-sigma-harm}. Together with the boundedness of $v^\sigma$ in $D_{\delta}$ and smooth boundary data $u$ on $\partial D_{\delta}$, the boundary regularity result (see for instance \cite[Theorem 8.33]{GT 2001}) implies
\begin{equation}\label{est-3}
\begin{split}
\|\nabla v^{\sigma}\|_{L^{\infty}(\partial D_{\delta})}&\leq C\left(\|v^{\sigma}\|_{L^{\infty}(D_{\delta})}+\|u\|_{W^{2,\infty}(\partial D_{\delta})}\right)\\
&\leq C\left(\|\phi-1\|_{L^{\infty}(D_{\delta})}+\|\psi+1\|_{L^{\infty}(D_{\delta})}+\|u\|_{W^{2,\infty}(\partial D_{\delta})}\right),
\end{split}
\end{equation}
where $C$ is independent of $\sigma$. By \eqref{eq:v-sigma-harm}, $\nabla v^{\sigma}$ is also harmonic in $D_{\delta}\setminus(\Lambda^{\sigma}_1\cup\Lambda^{\sigma}_2)$.
It follows from \eqref{est-1}, \eqref{est-2}, \eqref{est-3} and the maximum principle that
$$|\nabla v^{\sigma}|\leq C\quad\text{in}~D_\delta,$$
where $C$ is independent of $\sigma$. By Arzela-Ascoli theorem, there exists a function $v$ such that up to a subsequence as $\sigma\rightarrow0$, for any $0<\alpha<1$,
$$v^{\sigma}\rightarrow v\quad\text{in}~C^{0,\alpha}(\overline{D}_{\delta})$$
and
$$v^{\sigma}\rightarrow v\quad\text{weakly~in}~H^1(D_{\delta}).$$
Moreover, $\phi\leq v\leq\psi$ in $D_\delta$ and $v\in C^{0,1}(\overline{D}_{\delta})$.

To this end, we check that the limit function $v$ is actually $u$. By the minimality of $v^{\sigma}$, we have the variational inequality
$$\int_{D_{\delta}}\nabla v^{\sigma}\cdot\nabla(w-v^{\sigma})\,\mathrm{d}x\geq0,\quad\forall~w\in K^i_{\varepsilon}\subset K^{\sigma}.$$
Using the weak convergence in $H^{1}(D_{\delta})$, we have
$$\int_{D_{\delta}}\nabla v\cdot\nabla(w-v)\,\mathrm{d}x\geq0,\quad\forall~w\in K^i_{\varepsilon}.$$
Equivalently, $v$ minimizes $J^i_{\varepsilon}$ over $K^i_{\varepsilon}$. Note that there exists a unique minimizer of $J^i_{\varepsilon}$ in $K^i_{\varepsilon}$. The proof is finished by Lemma \ref{u-epsi}.
\end{proof}

\begin{remark}\label{rk:class-d-obstacle}
In the proof of Lemma \ref{Lip}, we actually prove Lipschitz continuity of solutions to general double obstacle problem with Lipschitz continuous obstacles. Specifically, let $U$ be a smooth and bounded domain in $\mathbb{R}^n$, and let $\phi,\psi\in C^{0,1}(\overline{U})$ and $g\in C^{1,1}(\overline{U})$ satisfying
$$\phi\leq\psi~\mathrm{in}~U\quad\mathrm{and}\quad\phi\leq g\leq\psi~\mathrm{on}~\partial U.$$
If $u$ is a minimizer of the energy functional $\int_{U}|\nabla v|^2\,\mathrm{d}x$ over the set
$$K_{\phi,\psi}:=\{v\in H^1(U)\mid\phi\leq v\leq\psi~\mathrm{a.e.~in}~U,~v=g~\mathrm{on}~\partial U\},$$
then $u\in C^{0,1}(\overline{U})$.
\end{remark}

Summing up, we now have all ingredients to present the proof of Theorem \ref{thm:main 1}.

\begin{proof}[Proof of Theorem \ref{thm:main 1}]
On the strength of Proposition \ref{small epsi}, the existence of a solution $u$ to problem \eqref{eq:P} is given by Lemma \ref{lem:P-e,existence}. By Lemma \ref{prop-2} (1), Proposition \ref{prop:u-e bdd spt} and Lemma \ref{Lip}, $u$ is Lipschitz continuous in $\mathbb{R}^n$. The regularity of $\partial\{u>0\}\setminus\overline{D}$ follows from Theorem \ref{properties} (5).
\end{proof}

Finally, combining Proposition \ref{small epsi} and Lemma \ref{u-epsi}, we prove Theorem \ref{thm:main 2}.

\begin{proof}[Proof of Theorem \ref{thm:main 2}]
By Proposition \ref{small epsi}, $u$ is also a solution to problem \eqref{eq:P-e} for any $0<\varepsilon<\varepsilon_0$.
Fix some $0<\varepsilon<\varepsilon_0$. We may restrict the consideration to a small ball $B_r(x_0)\subset D_{\delta_\varepsilon}$ with $x_0\in\Gamma$. After an appropriate translation and rotation, we can assume that $x_0=0$ and
$$B_r(x_0)\cap D=B_r^+\quad\text{and}\quad B_r(x_0)\cap\partial D\subset\Gamma.$$
Set $\mathcal{N}=D_{\delta_\varepsilon}$. By Lemma \ref{prop-2} (1), we have $u\in C^{1,1}_{\mathrm{loc}}(D)$. Combining \eqref{eq:Delta-u-e=0}, it suffices to prove
\begin{equation}\label{eq:u-C1,1/2}
u\in C^{1,\frac{1}{2}}\left(B_{\frac{r}{2}}^{\pm}\cup B'_{\frac{r}{2}}\right).
\end{equation}
We divide the rest of the proof into four steps.

Step 1. We prove that there exists a constant $C$ depending only on $||\phi||_{C^{1,1}(\overline{D})}$ and $||\psi||_{C^{1,1}(\overline{D})}$ such that
\begin{equation}\label{bdd-r}
|\Delta u|\leq C\quad\text{a.e.~in}~B_{\frac{3}{4}r}^{\pm}.
\end{equation}
Indeed, Lemma \ref{u-epsi} yields that $u$ solves the double obstacle problem in $B_r$ with $\phi_\varepsilon$ and $\psi_\varepsilon$ as obstacles. We thus get $u$ satisfies the Euler-Lagrange equation
\begin{equation}\label{E-L}
\Delta u=\Delta\phi\chi_{\{u=\phi\}}+\Delta\psi\chi_{\{u=\psi\}}
-\Delta\phi\chi_{\{\phi=\psi\}}\quad\text{in}~B_{\frac{3}{4}r}^{+}.
\end{equation}
By Lemma \ref{prop-2} (1), we have $u\in C_{\mathrm{loc}}^{1,1}(B_r^{+})$ and so \eqref{E-L} holds a.e. in $B_{\frac{3}{4}r}^{+}$.
On the other hand, it follows from \eqref{eq:Delta-u-e=0} that
$$\Delta u=0\quad\text{in}~B_{\frac{3}{4}r}^-,$$
and \eqref{bdd-r} is proved.

Step 2. We claim that $u$ solves single piecewise obstacle problems locally. Since $\phi_\varepsilon<\psi_\varepsilon$ on $B'_r$, together with the continuity of $\phi_\varepsilon$ and $\psi_\varepsilon$, then $$\phi_\varepsilon<\psi_\varepsilon\quad\text{in}~V,$$
where $V$ is a smooth neighborhood of $B'_r$ in $B_r$.
Then the coincidence sets
$$\Lambda_1:=V\cap\{u=\phi_\varepsilon\}\quad\text{and}\quad \Lambda_2:=V\cap\{u=\psi_\varepsilon\}$$
are disjoint. We may assume that $\Lambda_1$ and $\Lambda_2$ are nonempty, as otherwise the claim becomes trivial. It follows from Lemma \ref{u-epsi} that
\begin{equation}\label{lower ob}
\int_{V_1}|\nabla u|^2\,\mathrm{d}x=\min\left\{\int_{V_1}|\nabla v|^2\,\mathrm{d}x\mid v\in H^1(V_1),~v\geq\phi_\varepsilon~\text{in}~V_1,~v=u~\text{on}~\partial V_1\right\},
\end{equation}
\begin{equation}\label{upper ob}
\int_{V_2}|\nabla u|^2\,\mathrm{d}x=\min\left\{\int_{V_2}|\nabla v|^2\,\mathrm{d}x\mid v\in H^1(V_2),~v\leq\psi_\varepsilon~\text{in}~V_2,~v=u~\text{on}~\partial V_2\right\},
\end{equation}
where $V_1$, $V_2\subset V$ are smooth open sets satisfying
$$\Lambda_1\subset V_1\subset V\setminus\Lambda_2\quad\text{and}\quad\Lambda_2\subset V_2\subset V\setminus\Lambda_1.$$
By \eqref{lower ob}, $u_\varepsilon$ restricted to $V_1$ solves the single piecewise smooth obstacle problem studied in \cite{PT 2010} with $\phi_\varepsilon$ as obstacle. Also, by \eqref{upper ob}, $-u_\varepsilon$ restricted to $V_2$ solves such single obstacle problem with $-\psi_\varepsilon$ as obstacle.

Step 3. We show that $u(\cdot,0)$ is $C^{1,\frac{1}{2}}$ on $B_{\frac{r}{2}}'$. Denote $$\nabla'=(\partial_{x_1},\partial_{x_2},\cdots,\partial_{x_{n-1}}).$$
Since $u\in C^{0,1}(\overline{B}_r)$ by Lemma \ref{Lip}, we only need to check that
\begin{equation}\label{Holder 1/2}
|\nabla'u(x)-\nabla'u(y)|\leq C|x-y|^{\frac{1}{2}},\quad\forall~x,y\in B'_{\frac{r}{2}}.
\end{equation}
Denote
$$V'_i=V_i\cap B'_r~\mathrm{and}~\Lambda'_i=\Lambda_i\cap B'_r,\quad i=1,2.$$
We will take $V_1$ and $V_2$ in Step 2 such that
$$V'_1\subset\left\{x\in B'_r\mid \dist(x,\Lambda'_2)>\frac{d_0}{4}\right\},\quad V'_2\subset\left\{x\in B'_r\mid \dist(x,\Lambda'_1)>\frac{d_0}{4}\right\},$$
where $d_0=\dist\{\Lambda'_1,\Lambda'_2\}>0$. We also take
$$U_1=:\left\{x\in V'_1\mid \dist(x,\Lambda'_2)>\frac{d_0}{3}\right\},\quad U_2:=\left\{x\in V'_2\mid\dist(x,\Lambda'_1)>\frac{d_0}{3}\right\},$$
with $U_1\cup U_2=B'_r$.

To show \eqref{Holder 1/2}, we consider two cases:
\begin{enumerate}
\item [(1)] Suppose that $x,y\in U_1$. We may apply the local regularity theorem in \cite{PT 2010} to $u$ and $\phi_\varepsilon$ in $V_1$. It follows that
$$|\nabla'u(x)-\nabla'u(y)|\leq C|x-y|^{\frac{1}{2}}.$$
Likewise, if $x,y\in U_2$, \eqref{Holder 1/2} also follows.
\item [(2)] Suppose that $x\in U_1\setminus U_2$, $y\in U_2\setminus U_1$. Then $|x-y|\geq\frac{d_0}{3}$. Note that the local regularity theorem in \cite{PT 2010} implies that $|\nabla'u|$ is bounded in $B'_{\frac{r}{2}}$. Hence,
\begin{equation*}
\begin{split}
|\nabla'u(x)-\nabla'u(y)|&\leq|\nabla'u(x)|+|\nabla'u(y)|\\
&\leq C||\nabla'u||_{L^{\infty}\big(B'_{\frac{r}{2}}\big)}\left(\frac{|x-y|}{d_0}\right)^{\frac{1}{2}}\\
&\leq C|x-y|^{\frac{1}{2}}.
\end{split}
\end{equation*}
\end{enumerate}

Step 4. Combining Step 1 and Step 3, \eqref{eq:u-C1,1/2} follows from the boundary regularity theorem (see for instance \cite[Theorem 8.34]{GT 2001}).
This completes the proof.
\end{proof}

\section*{Acknowledgments}
The authors would like to thank professors Jiguang Bao and Hui Yu for their helpful suggestions and discussions. X. Li was supported by the China Scholarship Council (No. 202006040127). Part of this manuscript was written while X. Li was visiting the Department of Mathematics ``Federigo Enriques" at Universit\`{a} degli Studi di Milano, which is acknowledged for the hospitality.


\begin{thebibliography}{10}

\bibitem{AAC 1986} N.E. Aguilera, H.W. Alt, L.A. Caffarelli, An optimization problem with volume constraint, SIAM J. Control Optim. 24 (1986) 191–198.

\bibitem{ACS 1988} N.E. Aguilera, L.A. Caffarelli, J. Spruck, An optimization problem in heat conduction, Ann. Sc. Norm. Super. Pisa Cl. Sci. (5) 14 (1987) 355–387.

\bibitem{AC 1981} H.W. Alt, L.A. Caffarelli, Existence and regularity for a minimum problem with free boundary, J. Reine Angew. Math. 325 (1981) 105–144.

\bibitem{AFMT 1999} L. Ambrosio, I. Fonseca, P. Marcellini, L. Tartar, On a volume constrained variational problem, Arch. Ration. Mech. Anal. 149 (1999) 23–47.

\bibitem{BK 1973/74}  H. Br\'{e}zis, D. Kinderlehrer, The smoothness of solutions to nonlinear variational inequalities, Indiana Univ. Math. J. 23 (1973/74)  831–844.

\bibitem{C 1980} L.A. Caffarelli, Compactness methods in free boundary problems, Comm. Partial Differential Equations 5 (1980) 427–448.

\bibitem{C 1988} L.A. Caffarelli, The obstacle problem revisited, J. Fourier Anal. Appl. 4 (1998) 383–402.

\bibitem{CK 1980} L.A. Caffarelli, D. Kinderlehrer, Potential methods in variational inequalities, J. Anal. Math. 37 (1980) 285–295.

\bibitem{C 1979} M. Chipot, Sur la r\'{e}gularit\'{e} de la solution d'in\'{e}quations variationnelles elliptiques, C. R. Acad. Sci. Paris 288 (1979) 543-546.

\bibitem{C 1992} H.J. Choe, Regularity for certain degenerate elliptic double obstacle problems, J. Math. Anal. Appl. 169 (1992) 111-126.

\bibitem{DMV 1989} G. Dal Maso, U. Mosco, M.A. Vivaldi, A pointwise regularity theory for the two-obstacle problem, Acta Math. 163 (1989) 57-107.

\bibitem{FMW 2006} J. Fern\'{a}ndez Bonder, S. Mart\'{\i}nez, N. Wolanski, An optimization problem with volume constraint for a degenerate quasilinear operator, J. Differential Equations 227 (2006) 80–101.

\bibitem{FRW 2005} J. Fern\'{a}ndez Bonder, J.D. Rossi, N. Wolanski, Regularity of the free boundary in an optimization problem related to the best Sobolev trace constant, SIAM J. Control Optim. 44 (2005) 1612–1635.

\bibitem{FS 2019} A. Figalli, J. Serra, On the fine structure of the free boundary for the classical obstacle problem, Invent. Math. 215 (2019) 311–366.

\bibitem{GT 2001} D. Gilbarg, N.S. Trudinger, Elliptic partial equatios of second order, Classics in Mathematics, Springer-Verlag, Berlin, 2001. Reprint of the 1998 edition.

\bibitem{L 1996} C. Lederman, A free boundary problem with a volume penalization, Ann. Sc. Norm. Super. Pisa Cl. Sci. (4) 23 (1996) 249–300.

\bibitem{LT 2006} G.P. Leonardi, P. Tilli, On a constrained variational problem in the vector-valued case, J. Math. Pures Appl. (9) 85 (2006) 251–268.

\bibitem{L 1991 India} G.M. Lieberman, Regularity of solutions to some degenerate double obstacle problems, Indiana Univ. Math. J. 40 (1991) 1009–1028.

\bibitem{LFH 2016}F. Lin, Lectures on elliptic free boundary problems, in Lectures on the Analysis of Nonlinear Partial Differential Equations, Part 4, Morningside Lect. Math. 4, International Press, Somerville, MA, 2016, 115–193.

\bibitem{MZ 1991} J. Mu, W.P. Ziemer, Smooth regularity of solutions of double obstacle problem involving degenerate elliptic equations, Comm. Partial Differential Equations 16 (1991) 821–843.

\bibitem{PT 2010} A. Petrosyan, T. To, Optimal regularity in rooftop-like obstacle problem, Comm. Partial Differential Equations 35 (2010) 1292–1325.

\bibitem{SY 2021} O. Savin, H. Yu, On the fine regularity of the singular set in the nonlinear obstacle problem, Nonlinear Anal. 218 (2022), Paper No. 112770.

\bibitem{SY 2021A} O. Savin, H. Yu, Regularity of the singular set in the fully nonlinear obstacle problem, J. Eur. Math. Soc., to appear.

\bibitem{T 2005} E.V. Teixeira, The nonlinear optimization problem in heat conduction, Calc. Var. Partial Differential Equations 24 (2005) 21–46.

\bibitem{TU2017}
R. Teymurazyan, J.M. Urbano, A free boundary optimization problem for the $\infty$-Laplacian,
J. Differential Equations 263 (2017) 1140--1159.

\bibitem{Y 2016} H. Yu, An optimization problem in heat conduction with minimal temperature constraint, interior heating and exterior insulation, Calc. Var. Partial Differential Equations 55 (2016), Art. 130, 15 pp.

\end{thebibliography}
\end{document}